\documentclass{amsart}
\usepackage[all,2cell]{xy}
\usepackage{amssymb, amsmath, latexsym, a4}
\usepackage[marginratio=1:1,height=584pt,width=360pt,tmargin=117pt]{geometry}
\usepackage{color}

\usepackage{enumerate}
\UseTwocells
\def \Cat{\mathfrak{Cat}}
\def\im{\mathop\mathsf{Im}}
\def\cA{{\mathcal A}}

\newtheorem{De}{Definition}[section]
\newtheorem{Th}[De]{Theorem}
\newtheorem{Pro}[De]{Proposition}
\newtheorem{Le}[De]{Lemma}
\newtheorem{Co}[De]{Corollary}
\newtheorem{Rem}[De]{Remark}
\newtheorem{Ex}[De]{Example}
\def \Hom{\mathop\mathsf{Hom}\nolimits}

\def \cA{\mathcal{A}}
\def \cB{\mathcal{B}}
\def \cC{\mathcal{C}}
\def \cD{\mathcal{D}}
\def \cE{\mathcal{E}}
\def \cF{\mathcal{F}}
\def \cG{\mathcal{G}}
\def \cH{\mathcal{H}}

\def \cP{\mathcal{P}}

\def \cW{\mathcal{W}}

\def \fB{\mathfrak{B}}

\def \fF{\mathfrak{F}}

\def \F{\mathfrak{F}}
\def \fb{{\mathfrak B}}

\def \nn{{\bf n-1}}
\def \fb{{\mathfrak B}}

\def \c{{\sf c}}

\def \hom{\mathop\mathsf{Hom}\nolimits}
\def \Hom{\mathop\mathsf{Hom}\nolimits}

\def \Set{\mathop{\sf Set}\nolimits}
\def \colim{\mathop{\mathsf{colim}}}
\def \lim{\mathop{\mathsf{lim}}}
\def \tc{\mathop{2\textnormal{-}\mathsf{colim}}}
\def \tl{\mathop{2\textnormal{-}\mathsf{lim}}}

\def\Set{\mathsf{Set}}
\def \bk{{\bf k}}
\def \F{\mathfrak{F}}
\def \n{{\bf n}}
\def \nn{{\bf n-1}}

\def \tl{\mathop{2\textnormal{-}\mathsf{lim}}}

\def \Grp{\mathsf{Grpd}}

\def \cB{\mathcal{B}}
\def \cC{\mathcal{C}}
\def \cD{\mathcal{D}}

\def\xto#1{\xrightarrow[]{#1}}
\def\id{{\sf Id}}

\def\id{{\sf Id}}

\newcommand{\Z}{\mathbb{Z}}

\def\1{^{-1}}
\def \im{\mathop{\sf Im}\nolimits}

\def \ob{{\bf Ob}}
\def \mor{{\bf Mor}}

\def \hom{\mathop{\sf Hom}\nolimits}

\def\cP{\mathcal{P}}
\def\cW{\mathcal{W}}
\def\cE{\mathcal{E}}
\def \cG{\mathcal{G}}
\def \cF{\mathcal{F}}
\def \cC{\mathcal{C}}
\def \cH{\mathcal{H}}

\def\fB{\mathfrak{B}}
\def\fF{\mathfrak{F}}

\begin{document}
	
\title{Comparison of the Colimit and the 2-Colimit}

\author[I. Pirashvili]{Ilia Pirashvili}
\address{Universität Augsburg, Lehrstuhl ALgebra u. Zahlentheorie, Universitätsstraße 14, 86159 Augsburg}
\email{ilia.pirashvili@math.uni-augsburg.de}
\maketitle

\begin{abstract}
	The 2-colimit is the 2-categorical analogue of the colimit and as such, a very important construction. Calculating it is, however, more involved than calculating the colimit. The aim of this paper is to give a condition under which these two constructions coincide. Tough the setting under which our results are applicable is very specific, it is, in fact, fairly important: As shown in a previous paper, the fundamental groupoid can be calculated using the 2-colimit. The results of this paper corresponds precisely to the situation of calculating the fundamental groupoid from a finite covering.
	
	We also optimise our condition in the last section, reducing from exponential complexity to a polynomial one.
\end{abstract}

\hfill


\hfill

\emph{MSC classes}: 18A30, 18B40

\section{Introduction}\label{2colim}\label{chap.comp-col-2-col}

Let $X$ be a topological space and $\mathcal{U}=\{U_1,\cdots,U_n\}\in
\mathsf{Cov}(X)$ an open covering of $X$. We have shown in \cite{gmj} that under
some small assumptions on $X$, there are equivalences of categories
\[\Pi_1(X)\simeq\tc_{\fB({\bf n})}\Pi_1^*\simeq\colim_{\fB({\bf n})}\Pi_1^*.\]
Here, $\fB({\bf n})$ denotes the poset of proper subsets of $\{1,\cdots,n\}$ and
$\Pi_1^*:\fB({\bf n})\to \Grp$ is given by \[\Pi_1^*(J)=\Pi_1\left
(\bigcap_{j\not \in J} U_j\right)\] Of course, that $\Pi_1(X)$ is isomorphic to
$\colim\Pi_1^*$ is an easy consequence of the famous reformulation of the
classical Seifert-van Kampen theorem by Brown and Salleh
\cite{brown.sal},\cite{brown}.

However, showing that $\Pi_1(X)$ is equivalent to the 2-colimit, rather than the
colimit, has some advantages. For one, this allowed us to axiomatise the
fundamental groupoid in \cite{gmj} (and also the étale fundamental groupoid in
\cite{kyoto}). Additionally, it enables us to replace the participating
groupoids $\Pi_1(U_i)$ of the 2-colimit with equivalent ones, not just
isomorphic ones. This is very helpful for calculations,  as we shall show in
Example \ref{ex52504}. We can usually replace groupoids consisting of
uncountably many objects (which appears in Brown-Salleh formula) with groupoids
having only finitely many objects. The downside is that, calculating 2-colimits
is harder than calculating colimits. The aim of this work is to alleviate this.

Let $\fF:\fB({\bf n})\to\Grp$ be a strict 2-functor in the 2-category of small
groupoids. We will show in Theorem \ref{colim=2-colim} that the colimit and
2-colimit of $\fF$ are equivalent if a certain injectivity condition is
satisfied. Though this condition will seem fairly restrictive, it should be
pointed out that any 2-functor can be deformed to satisfy this condition. This
is due to the fact that injectivity will only be required on objects and it is
an obvious and well-known fact that any functor $\cF:\cG\to\cH$ between
groupoids (or categories) can easily be replaced with a functor
$\cF':\cG\to\cH'$ that is injective on objects, and for which $\cH$ and $\cH'$
are equivalent groupoids. As such, this is only a theoretical restriction that
will not be a great hindrance in practice.

We will give a few small applications of this result in calculating the
fundamental groupoid of a topological space. More applications can be found in a
joint paper with H. Brenner \cite{BP}.

Lastly, we will also show that it is possible to truncate the 2-colimit at
"depth 3". That is, we can define a new poset $\fB({\bf n,n-3})$ from $\fB({\bf
n})$ which consists of proper subsets of $\{1,\cdots, n\}$ with cardinality
$\geq n-3$. We have an equivalence of categories
\[\tc\fF\simeq\tc\fF',\]
where $\fF':\fB({\bf n,n-3})\to\Grp$ is the restriction of $\fF:\fB({\bf n})\to
\Grp$. We combine this with our previous result to simplify the condition under
which $\tc\fF\simeq\colim\fF$. In particular, the new condition is of polynomial
complexity.
 
The paper is organized as follows: In Section \ref{s.lim-colim} we recall the
construction of 2-limits and colimits for strict functors, defined over posets
with values in groupoids. We introduce the main trick in Section \ref{injob},
which enables us to replace 2-pushouts by the usual pushouts. Section
\ref{sec.bn} deals with one of the main result of this paper, Theorem
\ref{colim=2-colim}. The proof is based on the induction principle, which we
will prove in Section \ref{indprin}.

The conditions to check in the main theorem increases exponentially. Section
\ref{cod3} deals with reducing this to checking approximately $n^3$ conditions,
see Proposition \ref{62,25.04}.

\section{Construction of limits and colimits in $\Grp$}\label{s.lim-colim}

Let $I$ be a poset and $\Phi:I^{op}\to\Grp$ a strict contravariant functor to
the 2-category of groupoids. It is well-known \cite{gz}, that the category
$\Grp$ has limits and colimits, as well as 2-limits and 2-colimits. We aim to
recall these constructions in this section.

\subsection{Construction of limits in $\Grp$}\label{lim}

For $i\leq j$, let $\phi_{ij}:\Phi(j)\to\Phi(i)$ be the induced functor. Objects
of the groupoid $\lim_i\Phi$ (or simply $\lim\Phi$) are families $(x_i)$, such
that $x_i$ is an object of $\Phi(i)$ and  $\phi_{ij}(x_j)=x_i$ for any $i\leq
j$. A morphism $(x_i)\to(y_i)$ is a family $(f_i)$, such that $\phi_{ij}(f_j)=
f_i$ for any $i\leq j$. Here, $f_i:x_i\to y_i$ is a morphism of $\Phi(i)$.

\subsection{Construction of 2-limits in $\Grp$}

To construct the 2-limit $\tl\limits_i\Phi(i)$, we proceed as follows: Objects
of $\tl\limits_i\Phi(i)$ (or simply $\tl\Phi$) are collections $(x_i,\xi_{ij})$,
where $x_i$ is an object of $\Phi(i)$, while $\xi_{ij}:\phi_{ij}(x_j)\to x_i$
for $i\leq j$ is an isomorphism of the category $\Phi(j)$. It satisfies the so
called \emph{1-cocycle condition}: For any $i\leq j\leq k$, one has
\[\xi_{ik}=\xi_{ij}\circ\phi_{ij}(\xi_{jk}).\]
A morphism from $(x_i,\xi_{ij})$ to $(y_i,\eta_{ij})$ is a collection $(f_i)$,
where $f_i:x_i\rightarrow y_i$ is a morphism of $\Phi(i)$. Furthermore, the
following 
\[\xymatrix{\phi_{ij}(x_j)\ar[r]^{\xi_{ij}}\ar[d]_{\phi_{ij}(f_j)}&
	x_i\ar[d]^{f_i} \\  \phi_{ij}(y_j)\ar[r]_{\eta_{ij}}& y_i}\]
is a commutative diagram for any $i\leq j$.

Comparing the definitions of the limit and 2-limit, we see that we have the
following easy fact.
\begin{Le}\label{gamma.orsha}
    There is a full and faithful functor
    \[\gamma:\lim \Phi\rightarrow \tl \Phi,\]
    given by 
    \[\gamma(x_i)=(x_i,\id_{x_i}).\]
\end{Le}

\subsection{Construction of colimits in $\Grp$}\label{colim}

A sketch of the construction of $\colim_i\Psi(i)$ for a covariant functor
$\Psi:I\to\Grp$ is given as follows: (see \cite[pp.4-5, p.11]{gz}). The set of
objects of the groupoid $\colim_i\Psi(i)$ is the colimit of $i\mapsto
\ob(\Psi(i))$ in the category of sets, where $\ob(\Psi(i))$ is the set of
objects of $\Psi(i)$. Likewise, $\mor(\Psi(i))$ denotes the set of morphisms and
we take the colimit in the category of sets of the functor
$i\mapsto\mor(\Psi(i))$. The domain and codomain functions $\mor(\Psi(i))
\rightrightarrows\ob(\Psi(i))$ induce maps between the colimits
\[\colim_i\mor(\Psi(i))\rightrightarrows\colim_i\ob(\Psi(i)).\]
This yields a directed graph (\emph{diagram scheme} in the terminology of
\cite{gz}). Denote this directed graph by $X$ and let ${\bf Pa}(X)$ be the free
category generated by $X$. The category $\colim_i\Psi(i)$ is the maximal
quotient category for which the composite
\[\Psi(i)\to{\bf Pa}(X)\to\colim_i\Psi(i)\]
is a functor for all $i\in I$. It can be checked that the category
$\colim_i\Psi(i)$ is a groupoid, which is the colimit of $\Psi(i)$ in the
2-category $\Grp$.

\subsection{The Grothendieck construction}

Let $\Psi:I\to\Grp$ be a covariant 2-functor. We regard $\Grp$ as a natural
2-category.

Recall the Grothendieck construction $\int_I\Psi$. Objects of the category
$\int_I\Psi$ are pairs $(i,x)$, where $i$ is an object of $I$ and $x$ is an
object of $\Psi(i)$. A morphism $(i,x)\to (j,y)$ exists if $i\leq j$. It is
given by a morphism $\alpha:\psi_{i,j}(x)\to y$ of the category $\Psi(j)$. Let
$(i,x)$, $(j,y)$ and $(k,z)$ be objects of $\int_I\Psi$ with $i\leq j\leq k$.
The composition $(i,x)\to(k,z)$ of $(i,x)\to (j,y)\to (k,z)$, being given by
$\alpha:\psi_{i,j}(x)\to y$ and $\beta:\psi_{j,k}(y)\to z$ respectively, is
defined to be
\[\beta\circ\psi_{j,k}(\alpha):\psi_{i,k}(x)\to z.\]
We fix some additional notations. Let $i\in I$. Define the functor
\[L_i:\Psi(i)\to\int_I\Psi\]
by sending an object $x\in \ob(\Psi(i))$ to the pair $(i,x)$ and a morphism
$\alpha:x\to y$ to the morphism $(i,x)\to (i,y)$ determined by $\alpha$.

For $i\leq j$ there is a natural transformation
\[\lambda_{ij}:L_i \Rightarrow L_j\circ\Psi_{i,j},\]
defined by $\lambda_{ij}(x)= \id_{\psi_{ij}(x)}$. Here, $\id_{\psi_{ij}(x)}$ is
considered as a morphism $(i,x)\to(j,\psi_{i,j}(x))$. We observe that one has
\begin{equation}\label{eq1}
    \lambda_{i,k}=(\lambda_{j,k}*\psi_{jk})\circ\lambda_{i,j}
\end{equation}
for any $i\leq j\leq k$, where $*$ is the vertical composition of natural
transformations.

\subsection{The Grothendieck construction related to the 2-limit}

Equation (\ref{eq1}) allows us to define the functor
\[K:\Hom_\Cat\left(\int_I\Psi,\cG\right)\to\tl_i\Hom_{\Grp}(\Psi(i),\cG)\]
for any groupoid $\cG$, by setting
\[K(\theta)=(\theta\circ L_i,\theta\circ\lambda_{ij}).\]
Here, $\Cat$ is the 2-category of small categories, $\theta:\int_I\Psi\to\cG$
is a functor and $\Hom_\Cat$ denotes the category of functors.

The functor $K$ has an inverse
\[J:\tl_i\Hom_\Grp(\Psi(i),\cG)\to\Hom_\Cat\left(\int_I\Psi,\cG\right),\]
which is constructed as follows: Recall that an object of $\tl_i\Hom_\Grp
(\Psi(i),\cG)$ is a collection $(X_i,A_{ij})$, where $X_i:\Psi(i)\to\cG$ is a
functor and for $i\leq j$, $A_{ij}:X_i\Rightarrow X_j\circ\psi_{i,j}$ is a
natural transformation satisfying the 1-cocycle condition. Define the functor
\[J(X_i,A_{ij}):\int_I\Psi\to\cG\]
on objects by
\[(i,x)\mapsto X_i(x)\]
and on morphisms by
\[\alpha\mapsto\alpha:X_j(\alpha)\circ A_{ij}(x).\]
Here, $\alpha:(i,x)\to(j,y)$ is a morphism in $\int_I\Psi$, induced by
$\alpha:\psi_{ij}(x_i)\to x_j$.

This makes $K$ an equivalence of categories. In other words, the Grothendieck
construction can be considered a model of the 2-colimit in $\Cat$, the 2-colimit
of categories. If the Grothendieck construction would have been a groupoid, it
would have been the 2-colimit in the 2-category of groupoids. It is in general,
however, only a category. We will remedy this shortly.

\subsection{The Grothendieck construction related to the 2-colimit}

Let $\iota:\Cat\to\Grp$ be the left adjoint of the inclusion functor $\Grp
\subseteq\Cat$. Recall that $\iota(\cC)$ is obtained from $\cC$ by inverting
every morphisms in $\cC$ (see, \cite[Section 1.5.4]{gz}). In the case when $\cC
=\int_I\Psi$, it suffices to invert the morphisms $(i,x)\to(j,\psi_{ij}(x_i))$,
defined by $\id_{\psi_{ij}(x)}$. It follows that the groupoid
$\iota(\int_I\Psi)$ is the 2-colimit of $\Psi$.

\subsection{Comparison functor}\label{272504}

The 2-categorical version of the Yoneda Lemma and the functor $\gamma:\lim\Phi
\to\tl\Phi$, defined in Lemma \ref{gamma.orsha}, allow us to define a comparison
functor
\[\delta:\tc\Psi\to\colim\limits\Psi.\]
This represents the following composition of functors:
\[\Hom_\Grp(\colim\Psi,\cG)\cong\lim\Hom_\Grp(\Psi,\cG)\xto{\gamma}
\tl\Hom_\Grp(\Psi,\cG)\cong\Hom_\Grp(\tc\Psi,\cG).\]
We wish to understand the following question: Under what conditions on $\Psi$ is
$\delta$ an equivalence of groupoids?

\section{Injectivity on objects}\label{injob}

To fix terminology, let $\cF:\cG\to\cH$ be a functor between small groupoids.
This induces a map $\ob(\cF):\ob(\cG)\to\ob(\cH)$ from the set of objects of
$\cG$ to the set of objects of $\cH$. We say that the functor $\cF$ is
\emph{injective on objects} if the map $\ob(\cF)$ is injective.

\subsection{The initial step}

\begin{Le}[Deformation Lemma]\label{deformation}
    Let 
    \[\xymatrix{\cA\drtwocell\omit{\lambda}\ar[r]^{i_1}\ar[d]_{i_2 }&
    \cB\ar[d]^{j_1} &\\ \cC\ar[r]_{j_2} & \cD & }\]
    be a $2$-diagram of groupoids (meaning, $\lambda:j_1i_1\Rightarrow j_2i_2$
    is a natural isomorphism). If the functor $i_1$ is injective on objects,
    there exist a functor $j_1':\cB\to\cD$ and a natural transformation
    $\kappa:j_1\Rightarrow j_1'$, for which the diagram
    \[\xymatrix{\cA\ar[r]^{i_1}\ar[d]_{i_2} & \cB\ar[d]^{j_1'}\\
    \cC \ar[r]_{j_2} & \cD}\]
    commutes and $\lambda$ coincides with $\kappa\star i_1:j_1i_1\Rightarrow
    j_1'i_1=j_2i_2.$
\end{Le}

\begin{proof}
	Define $j_1'$ on objects by
	\[j_1'(b)=\begin{cases}j_1(b), & {\rm if} \ b\not\in \im(i_1)\\
	j_2i_2(a),& {\rm if} \ b=i_1(a).\end{cases}\]
	Let $\beta:b_1\rightarrow b_2$ be a morphism in $B$. To define
	\[j_1'(\beta):j_1'(b_1)\rightarrow j_1'(b_2),\]
	we have to consider five different cases.
	
    \noindent
    \underline{Case 1}: $b_1=i_1(a_1)$ and $b_2\not\in\im(i_1)$. Define
    $j_1'(\beta)$ to be the composite:
	\[j_1'(b_1)=j_2i_2(a_1)\xto{\lambda^{-1}(a_1)}j_1i_1(a_1)=
	j_1(b_1)\xto{j_1(\beta)} j_1(b_2)=j_1'(b_2).\]
	
    \noindent
    \underline{Case 2}: $b_1=i_1(a_1)$, $b_2=i_1(a_2)$, but $\beta\not \in
    \im(i_1)$. Define $j_1'(\beta)$ to be the composite:
	\[j_1'(b_1)=j_2i_2(a_1)\xto{\lambda^{-1}(a_1)}j_1i_1(a_1)=j_1(b_1)
	\xto{j_1(\beta)}j_1(b_2)=j_1i_1(a_2)\xto{\lambda(a_2)}j_2i_2(a_2)=j_1'(b_2).\]
	
    \noindent
	\underline{Case 3}: $\beta\in\im(j_1)$. Choose $\alpha:a_1\to a_2$
	in $A$ such that $\beta=i_1(\alpha)$ and define $j_1'(\beta)$ by
	\[j_1'(b_1)=j_2i_2(a_1)\xto{j_2i_2(\alpha)} j_2i_2(a_2)=j_1'(b_2).\]
	To check that this is independent of the choice of $\alpha$, consider
	$\alpha':a_1\to a_2$ with the property $i_1(\alpha')=\beta$. Since $\lambda$
	is a natural transformation, we have a commutative diagram
	\[\xymatrix{j_1i_1(a)\ar[rr]^{\lambda(a)}\ar@<-1ex>[d]_{_{j_1i_1(\alpha)}}
		\ar@<1ex>[d]_{\ \ =}\ar@<1ex>[d]^{_{j_1i_1(\alpha')}} && j_2i_2(a)
		\ar@<-1ex>[d]_{_{j_2i_2(\alpha)}}\ar@<1ex>[d]^{_{j_2i_2(\alpha')}} \\
		j_1i_1(a')\ar[rr]_{\lambda(a')} & & j_2i_2(a') }\]
	Since the left vertical arrows are equal and the horizontal ones are
	isomorphisms, it follows from the commutativity, that the right vertical
	arrows are also equal. Hence, $j_1'(\beta)$ is well-defined in this case.
	
    \noindent
    \underline{Case 4}: $b_1\not\in\im(i_1)$ and $b_2=i_1(a_2)$. Define
    $j_1'(\beta)$ to be the composition
	\[j_1'(b_1)=j_1(b_1)\xto{j_1(\beta)} j_1(b_2)=j_1
		i_1(a_2)\xto{\lambda(a_2)} j_2i_2(a_2)=j_2'(b_2).\]
	
    \noindent
	\underline{Case 5}: $b_1\not\in\im(i_1)$ and $b_2\not\in\im(i_1)$.
	Define $j_1'(\beta)$ as
	\[j_1'(b_1)=j_1(b_1)\xto{j_1(\beta)} j_1(b_2)=j_1'(b_2).\]
	Checking case by case shows that $j_1'$ is really a functor, with
	$j_1'i_1=j_2i_2$. Define $\kappa$ by
	\[\kappa(b)=\begin{cases} \id_{j_1(b)}, & {\rm if} \ b\not\in\im(i_1)\\
	\lambda(a),& {\rm if} \ b=i_1(a). \end{cases}\]
	One easily sees that $j_1'$ and $\kappa$ satisfy the assertions of the
	Lemma.
\end{proof}

The following result already appeared in \cite{gmj} as Proposition 4.2.

\begin{Th}\label{2}
    Let
	\[\xymatrix{ \cA\ar[r]^{i_1}\ar[d]_{i_2} & \cB \\ \cC & }\]
    be a diagram of groupoids, where $i_1$ is injective on objects. The colimit
    and the 2-colimit are equivalent.
\end{Th}

\begin{proof}
    Consider the diagram 
	\[\xymatrix{ \cA\ar[r]^{i_1}\ar[d]_{i_2} & \cB\ar[d]\ar[ddrr] & & \\ 
		\cC\ar[r]\ar[drrr] & \cP &&\\ &&& \cW\ar[ull]_{\delta}, }\]
    where $\cP$ is the pushout and $\cW$ is the 2-pushout of
    $\cC\xleftarrow{i_2}\cA\xto{i_1}\cB$. Let $\cE$ be any groupoid and
    consider:
	\[\xymatrix{ \hom(\cA,\cE) & \hom(\cB,\cE)\ar[l]^{i_1^*} &&
		\\\hom(\cC,\cE)\ar[u]_{i_2^*}&\hom(\cP,\cE)\ar[u]\ar[l]
		\ar[drr]^{\delta^*}&&\\&&& \hom(\cW,\cE).\ar[ulll]\ar[uull]}\]
	Since $\hom(-,\cE)$ maps pushouts and 2-pushouts to pullbacks and
    2-pullbacks, $\hom(\cP,\cE)$ and $\hom(\cW,\cE)$ are the pullback,
    respectively 2-pullback, of the diagram
	\[\hom(\cB,\cE)\xto{i_1^*}\hom(\cA,\cE)\xleftarrow{i_2^*}\hom(\cC,\cE).\]
    Lemma \ref{gamma.orsha} implies that $\delta^*$ is full and faithful. It
    remains for us to show that $\delta^*$ is also essentially surjective. 
	
    Since, as mentioned, $\hom(\cW,\cE)$ is the 2-pullback, taking an object
    there is the same as taking objects in $\hom(\cB,\cE)$ and $\hom(\cC,\cE)$,
    and an equivalence between these objects in $\hom(\cA,\cE)$. That is to say
    we, have a 2-commutative diagram
	\[\xymatrix{ \cA\ar[r]^{i_1}\ar[d]_{i_2} &
		\cB\ar[d] \ \\  \cC\ar[r] & \cE.\ultwocell\omit{} }\]
    From this, we immediately see that the Deformation Lemma (\ref{deformation})
    shows the essential surjectivity of $\delta^*$. This in turn, using the
    Yoneda lemma for 2-categories, shows that $\delta:\cW\to\cP$ is an
    equivalence of categories, completing the proof.
\end{proof}

\section{The Poset $\fB(\n)$}\label{sec.bn}

Denote by $\fB(V)$ the poset of all proper subset of a set $V$. By proper we
mean that the subset should not equal $V$ itself, it can, however, be empty.
More generally, for a proper subset $U\subsetneq V$, we set
\[\fB(V:U)=\{X\in \fB(V)\ |\ U\subseteq X\}.\]
Thus, $\fB(V:\emptyset)=\fB(V)$ and $\fB(V:V)=\emptyset$. If $V \subseteq W$,
then $\fB(V:U)\subseteq \fB(W:U)$. Hence, we can and we will identify the
elements of $\fB(V:U)$ with the corresponding elements of $\fB(W:U)$.

We use the notation $\n:=\{1,2,\cdots,n\}$, where $n\geq 1$. Accordingly, we
write $\fB(\n)$ instead of $\fB(\{1,\cdots,n\})$. If $k<n$, $\fB(\bk)$ can be
identified with the subposet of $\fB(\n)$, consisting of elements
$X\subsetneq\n$ such that $k+1,\cdots, n\not \in X$ and $X\not=\{1,\cdots,k\}$.
In particular, $\fB({\bf n-1})$ can be identified with the subposet of $\fB(\n)$
consisting of subsets $X\subsetneq\{1,\ldots,n-1\} \subseteq\n$, such that
$n\not\in X$ and $X\neq\{1,\cdots,n-1\}$. We can also immediately see that
\[\fB(\n:{\bf m})=\{X\,|\{1,\ldots,m\}\subseteq
X\subsetneq\{1,\ldots,n\}\}\simeq \fB({\bf n-m}).\]

Parallel to the above, we consider the subposet $\fB'({\bf n-1}): =
\fB(\n:\{n\})$. It consists of subsets $Y\subsetneq\n$, such that $n \in
Y$. This gives us a partition of posets
\[\fB(\n)=\fB({\bf n-1})\coprod \fB'({\bf n-1})\coprod\{\bf n-1\}.\]
Obviously, the map $\alpha:\fB({\bf n-1})\to\fB'({\bf n-1})$, given by
$X\mapsto X\coprod \{n\}$, defines an isomorphism of posets.
	
Let $U$ be a proper subset of $V\in\fB(\n)$. We will say that $\Psi$ satisfies
the condition $A^V_U$ if the canonical functor
\[\colim_{\fB(V:U)}\Psi\to\Psi(V)\]
is injective on objects. 

Let $n\geq 2$ be an integer. We say that $\Psi$ satisfies the condition $A({n})$
if for any proper subset $1\in V\subsetneq {\bf n}$, condition $A^V_U$ holds.
Here, $U$ is constructed from $V$ as follows: Let $k$ be the minimal integer
such that $k+1\not \in V$. Set $U=V\setminus{\bf k}$.

\begin{Th}\label{colim=2-colim}
	Let $n\geq 2$ be a natural number and $\Psi:\fB(n)\to\Grp$ a covariant
	strict $2$-functor. The comparable functor
	\begin{eqnarray}\label{eq.delta}
        \delta_{\fB(n)}:\tc_{\fB(n)}\Psi\to\colim_{\fB(n)}\Psi
	\end{eqnarray}
	is an equivalence of categories if condition $A(n)$ holds.
\end{Th}

\begin{proof}
    We proceed by induction. Assume $n=2$. Then $V=\{1\}$ and thus
    $U=\emptyset$. By our assertion, the canonical functor
	$\Psi(\emptyset)\to\Psi({\bf 1})$ is injective on objects. Hence, the case
	$n=2$ corresponds to Theorem \ref{2}. The general case follows from $n=2$
	and Corollaries \ref{1nab} and \ref{2nab}, proved below.
	
    In more detail, assume $n>2$. We have the 2-pushout diagram of groupoids
	\[\xymatrix{ \tc_{\fB(\nn)}\Psi\ar[r]\ar[d] & \tc_{\fB'(\nn)}\Psi\ar[d]\\
    \Psi({\bf n-1})\ar[r] & \tc_{\fB(\n)}\Psi}\]
    thanks to Corollary \ref{2nab}. By Corollary \ref{1nab}, we have the pushout
    diagram
	\[\xymatrix{\colim_{\fB(\nn)}\Psi\ar[r]\ar[d]&\colim_{\fB'(\nn)}\Psi\ar[d]
    \\ \Psi({\bf n-1})\ar[r] & \colim_{\fB(\n)}\Psi. }\]
    Observe that $\delta$ (see Section \ref{272504}) maps the first square to
    the second square. The left vertical map in the bottom diagram corresponds
    to our condition when $V={\bf n-1}$, hence is injective on objects. It
    follows again from Theorem \ref{2} that the second square is also a
    2-pushout diagram. To show that $\delta_{\fB(\n)}$ is an equivalence of
    groupoids, it suffices to show that $\delta_{\fB(\nn)}$ and
    $\delta_{\fB'(\nn)}$ are equivalences.
	
    One can easily see that condition $A(n)$ implies $A(n-1)$. So, our theorem
    also holds for $\fB(\nn)$. Hence, the functor $\delta_{\fB(\nn)}$ is an
    equivalences of categories by our induction assumption.
	
	The map $\alpha:\fB({\bf n-1})\to\fB'({\bf n-1})$ is an isomorphism of
    posets. We wish to apply the induction assumption to the poset $\fB'(\nn)$.
    For this, we have to transfer condition $A(n-1)$ to condition $A'(n-1)$,
    using the bijection $\alpha$.
	
    The restated condition becomes $A^{V'}_{U'}$, where $\{1,n\}\subseteq V'
    \subsetneq {\bf n}$. It is clear that our initial condition encompasses
    these cases as well. To see this, observe that such $V'$ can be written as
    $V'=V\cup \{n\}$, $1\in V$ and if $k$ is the minimal integer for which
    $k+1\not \in V$, the same will hold for $V'$. Hence, $U'=U\cup \{n\} =
    V'\setminus \bf k$. This proves that $\delta_{\fB'(\nn)}$ is also an
    equivalence of categories and we are done.
\end{proof}

While the condition in the above theorem seem fairly strict, injectivity on
objects is actually not a big hurdle. Any functor can be replaced with an
equivalent one, such that the conditions of the theorem are satisfied. We call
this \emph{stretching} and will discuss it below.

\begin{Rem}
	The conditions of Theorem \ref{colim=2-colim} can be simplified if the
	functors $\F(U_\alpha)\to\F(U_\beta),\alpha\leq\beta$ are surjective on
	objects. In such a case, we only need to ensure that the functors
	$\F(U_\alpha)\to\F(U_\beta)$ are injective on objects. This will imply
	bijectivity on objects of the functors $\F(U_\alpha)\to \F(U_\beta)$, and as
	such, $\colim_{\fb^\c(I:J)}\Psi\to\Psi(I^\c)$ will also be bijective on
	objects.
\end{Rem}

\subsection{Stretching}

Let $\cF:\cG\to\mathcal{H}$ be a functor. Assume there exist $x\neq
x'\in\cG$, such that $\cF(x)=\cF(x')=y$. Define $\mathcal{H}'$ to be the
category generated by adding a new object $y'$ to $\mathcal{H}$ and an
isomorphism $\alpha:y\rightarrow y'$. Let $z,z'$ be objects of $\mathcal{H}'$.
We have
\[\Hom_{\mathcal{H}'}(z,z')=
    \begin{cases}
        \Hom_{\mathcal{H}}(z,z'), & z\not =y'\not =z'\\
        \Hom_{\mathcal{H}}(y,z')\alpha^{-1}, & z =y'\not =z'\\
        \alpha \Hom_{\mathcal{H}}(z,y), & z \not=y' =z'\\
        \alpha \Hom_{\mathcal{H}}(y,y)\alpha^{-1}, & z =y'=z'.
    \end{cases}
\]
We also define a new functor $\cF':\cG\to\mathcal{H}'$ as follows: On
objects, we set
\[\cF'(u)=
    \begin{cases}
        \cF(u), & u\not =x'\\ y,& u=x'.
    \end{cases}
\]
Here, $u$ is an object of $\mathcal{H}$. On morphisms, $\cF'$ is defined by
\[\cF'(u\xto{\beta}v)=
    \begin{cases}
        \cF(\beta), & u\not =x'\not =v\\
        \cF(\beta)\alpha^{-1},& u=x'\not = v'\\
        \alpha\cF(\beta),& u\not =x'= v'\\
        \alpha\cF(\beta)\alpha^{-1},& u=x' = v'.
    \end{cases}
\]
The two functors $\cF$ and $\cF'$ are isomorphic via $\theta:\cF\Rightarrow
cF'$, where
\[\theta(u)=
    \begin{cases}\id_{\cF(u)}, & u\not =x'\\ \alpha,& u=x'.
    \end{cases}
\]
We call this construction \emph{stretching}\label{stretching}.

A trivial consequence of our main theorem says that the canonical map
$\delta_{\fB(n)}:\tc_{\fB(n)}\Psi\to\colim_{\fB(n)}\Psi$ is an equivalence of
categories if condition $A^V_U$ holds for all subsets $U\subsetneq
V\subsetneq\n$. As such, we do not loose much by restating the theorem in a
slightly less general, but more intuitive way:

\begin{Co}\label{colim=2-colim.Co}
	Let $n\geq 1$ be a natural number and $\Psi:\fB(\n)\to\Grp$ a covariant
    strict 2-functor. Assume that for all $S\in\fB(\n)$, the functor
    $\colim_{\fB(S)} \Psi\to\Psi(S)$ is injective on objects. The canonical
    functor
	\[2\textnormal{-}\colim_{\fB(\n)}\Psi\to\colim_{\fB(\n)}\Psi\]
	is an equivalence of categories.
\end{Co}
	
We can see the application of the above theorem (in conjunction with \cite[Thm
3.2]{gmj}) already in the following simple example:

\begin{Ex}\label{ex52504}
    Let $x,y\in S^1$ be two distinct points on the unit circle $S^1$. The
    subsets $U_x:=S^1\setminus\{x\}$ and $U_y:=S^1\setminus\{y\}$ define an open
    covering. Set $U_{xy}:=U_x\cap U_y$. By \cite[Thm 3.2]{gmj}, we know that
    $\Pi_1(S^1)$ can be expressed as the 2-pushout associated to the above
    covering. Since the fundamental groupoids in this 2-pushout have no
    non-trivial morphisms, it suffices to give this diagram (including the
    functors) purely on objects. We represent it as follows:
	\[\xymatrix@C=0.2em{ \bullet &&&&&&&&& \bullet
	\\  &&&& \bullet\ar[ullll]\ar[urrrrr]&\bullet\ar[ulllll]\ar[urrrr].&&&&}\]
    Here, $\bullet$ denotes the groupoid with one object and
    $\bullet\quad\bullet$ denotes the groupoid with two objects. In both cases,
    we have only identity morphisms. The arrows depict the functors on
    objects.
	
    As $n=2$, the condition of Theorem \ref{colim=2-colim} only requires for the
    functor $\Pi_1(U_{12})\to\Pi_1(U_2)$ to be injective on objects. As the
    2-colimit is invariant under equivalences, we can replace our diagram with
	\[\xymatrix@C=0.2em{ \bullet &&&&&&&&&& \bullet\ar@{-}@/^0.5pc/[r] & \bullet
	\\&&&&\bullet\ar[ullll]\ar[urrrrrr]&\bullet\ar[ulllll]\ar[urrrrrr].&&&&}\]
    Here, $\xymatrix@C=0.5em{\bullet\ar@{-}@/^0.5pc/[rr]& & \bullet}$ depicts
    the groupoid with two objects and a single isomorphism between them, which
    is clearly equivalent to the groupoid $\bullet$. As such, $\Pi_1(S^1)$ is
    equivalent to the colimit of the second diagram, which can easily be
    calculated to be the groupoid with one object and automorphism group $\Z$.
\end{Ex}

\section{Induction principle for colimits and 2-colimits}\label{indprin}

\subsection{Induction principle for colimits}

The aim of this section is to prove Corollary \ref{1nab} below, which we used in
our proof of Theorem \ref{colim=2-colim}. Let $\Phi:\fB(\n)^{op}\to\Set$ be a
contravariant functor. According to the definition, elements of
$\lim_{\fB(\n)}\Phi$ are families $(a_W)$, satisfying the compatibility
condition
\[\Phi_{V,W}(a_W)=a_V, \quad V\subseteq W.\]
Here, $W$ and $V$ are running through all the proper subsets of $\n$.

We would like to have an inductive procedure for studying such limits. Consider
the restriction of $\Phi$ on $\fB(\nn)$ and on $\fB'(\nn)$. By abuse of
notation, we denote these functors by $\Phi$ as well. Consider the limit of
$\Phi$ over these posets. This yields the diagram
\[\xymatrix{ & \lim_{\fB'(\nn)}\Phi\ar[d]^\eta\\
	\Phi({\bf n-1})\ar[r]_\delta & \lim_{\fB(\nn)}\Phi. }\]
Here, $\delta$ sends $a_{\bf n-1}\in \Phi({\bf n-1})$ to the family 
$(\Phi_{U,{\bf n-1}}(a_{\bf n-1}))$, where $U$ is a proper subset of ${\bf
n-1}$. In order to describe $\eta$ for a proper subset $U\subsetneq {\bf n-1}$,
we set
\[U_+ =\{n\}\cup U.\]
The functor $\eta$ sends $(a_{U_+})\in
\lim_{\fB'(\nn)}\Phi$ to $(\Phi_{U, U_+}(a_{U_+}))\in \lim_{\fB(\nn)}\Phi$.

\begin{Le}
	The following diagram is a pullback in the category of sets:
    \[\xymatrix{ \lim_{\fB(\n)}\Phi \ar[r]\ar[d] & \lim_{\fB'(\nn)}\Phi\ar[d]
	\\  \Phi({\bf n-1})\ar[r] & \lim_{\fB(\nn)}\Phi. }\]
\end{Le}

\begin{proof}
	Denote the pullback of the diagram
	\[\xymatrix{ & \lim_{\fB'(\nn)}\Phi\ar[d]^\eta\\
		\Phi({\bf n-1})\ar[r]_\delta & \lim_{\fB(\nn)}\Phi }\]
	by $P$. By definition elements of $P$ are pairs $(b_{\bf n-1}, B)$,
    satisfying the condition
    \[\delta(b_{\bf n-1})=\eta(B).\]
    Here, $b_{\bf n-1}$ denotes an element in $\Phi({\bf n-1})$, while $B\in
    \lim_{\fB'(\nn)}\Phi$. That is to say, $B$ is itself a compatible family
    $(b_{U_+})$, where $U$ is a proper subset of $\bf n-1$ and $b_{U_+}$ is an
    object of $\Phi(U_+)$. These objects satisfy the compatibility condition
	\[\Phi_{U_+,V_+}(b_{V_+})=b_{U_+},\]
	where $U\subseteq V$ are proper subsets of ${\bf n-1}$. The condition
	$\delta(b_{\bf n-1})=\eta(B)$ means that one has
	\[\Phi_{U,{\bf n-1}}(b_{\bf n-1})=\Phi_{U,U_+}(b_{U_+})\]
	for any proper subset $U\subseteq {\bf n-1}$. Define the map
	$\xi:\lim_{\fB(\n)}\Phi\to P$ by
	\[(a_W)\mapsto (a_{\bf n-1},A);\quad A=(a_{U_+}).\]
	Here, $W$ (resp. $U$) is running through the set of proper subset of $\n$
	(resp. ${\bf n-1}$). We need to show that $\xi$ is a bijection. To this end,
	define the inverse map $\theta:P\to\lim_{\fB(\n)}\Phi$ by
	\[(b_{\bf n-1}, B)\mapsto (a_W).\]
	Here, we denote
	\[a_W=\begin{cases}
    b_W,\quad{\rm if}\quad W\in \fB'(\nn)
	\\  b_{\bf n-1},\quad{\rm if}\quad W={\bf n-1}
	\\  \Phi_{W,W_+}(b_{W_+}),\quad{\rm if}\quad W\subseteq \fB(\nn).
	\end{cases}\]
	One readily shows that $(a_W)\in \lim_{\fB(\n)}\Phi$ and that $\theta$ is
	inverse to $\xi$.
\end{proof}

Yoneda's lemma yields the following statement for 2-colimits:

\begin{Co}\label{1nab}
	Let $\cC$ be a category with finite limits (resp. colimits) and let
	$\Phi:\fB(\n)^{op}\to\cC$ (resp. $\Psi:\fB\bf(n)\to\cC$) be a
	contravariant (resp. covariant) functor. The following is a pullback diagram
	in the category $\cC$:
	\[\xymatrix{\lim_{\fB(\n)}\Phi\ar[r]\ar[d] & \lim_{\fB'(\nn)}\Phi\ar[d]
	\\  \Phi({\bf n-1})\ar[r]&\lim_{\fB(\nn)}\Phi. }\]
	Likewise, the following is a push-out diagram in the category $\cC$:
	\[\xymatrix{\colim_{\fB(\nn)}\Psi\ar[r]\ar[d]&\colim_{\fB'(\nn)}\Psi\ar[d]
	\\  \Psi({\bf n-1})\ar[r] & \colim_{\fB(\n)}\Psi.}\]
\end{Co}

\subsection{Induction principle for 2-colimits}

The aim of this section is to prove Corollary \ref{2nab}, also used in the proof
of Theorem \ref{colim=2-colim}. Let $\Phi:I^{op}\rightarrow \Grp$ be a
contravariant functor from the category $I$ to the category of groupoids.

Recall the construction of the groupoid $\tl\limits_i\Phi_i$, as discussed in
Section \ref{s.lim-colim}. The 2-pullback (which is in particular a 2-limit)
of the diagram
\[\xymatrix{ & \cG_1\ar[d]^{f_1} \\ \cG_2\ar[r]_{f_2}& \cG_0}\]
is the following groupoid: Objects are triples $(A_1,A_2,\alpha)$, where $A_i$
is an object of $\cG_i$ and $\alpha:f_1(A_1)\to f_2(A_2)$ is an isomorphism. A
morphism $(A_1,A_2,\alpha)\to(B_1,B_2,\beta)$ is a pair $(g_1,g_2)$, where
$g_i:A_i\to B_i$ is a morphism in $\cG_i$, such that the following diagram
\[\xymatrix{ f_1(A_1)\ar[r]^{\alpha}\ar[d]_{f_1(g_1)}&f_2(A_2)\ar[d]^{f_2(g_2)}
\\  f_1(B_1)\ar[r]_\beta& f_2(B_2) }\]
commutes.

Likewise, objects of the groupoid $\tl\limits_{\fB(\n)}\Phi$ are collections
\[(a_U,\alpha_{U,V}:\Phi_{U,V}(a_V)\to a_U),\]
where $U$ and $V$ are proper subsets of $\n$ with $U\subseteq V$. Moreover,
$a_U$ (resp. $\alpha_{U,V}$) is an object (resp. isomorphism) of $\Phi(U)$. One
requires that for all $U\subseteq V\subseteq W\subsetneq {\bf n}$, the following
diagram
\[\xymatrix{\Phi_{U,W}(a_W)\ar[rr]^{\Phi_{V,W}(\alpha_{VW})}
\ar[dr]_{\alpha_{UW}}&& \Phi_{U,V}(a_V)\ar[dl]^{\alpha_{U,V}}\\	& a_U & }\]
commutes. A morphism $(a_U,\alpha_{U,V})\to(b_U,\beta_{U,V})$ is a collection
$(g_U)$, where $U$ is a proper subset of $\n$ and $g_U:a_U\to b_U$ is a morphism
of $\Phi(U)$. Further, for all proper subsets $U\subseteq V$, the following
diagram
\[\xymatrix{ \Phi_{U,V}(a_V)\ar[r]^{\alpha_{U,V}}\ar[d]_{\Phi_{U,V}(g_V)} & a_U
\ar[d]^{g_U}\\ \Phi_{U,V}(b_V)\ar[r]^{\beta_{U,V}} & b_U }\]
has to commute. Let us denote by $\cP$ the 2-pullback of the diagram
\[\xymatrix{ & \tl\limits_{\fB'(\nn)}\Phi\ar[d]^\eta
\\  \Phi({\bf n-1})\ar[r]_\delta & \tl\limits_{\fB(\nn)}\Phi.}\]
The functor $\delta$ sends an object $a_{\bf n-1}$ of the category 
$\Phi({\bf n-1})$ to the collection $(a_U,\alpha_{U,V})$, where $U\subseteq V$
are proper subsets of $\bf n-1$, $a_U=\Phi_{U,{\bf n-1}}(a_{\bf n-1})$ and
$\alpha_{U,V}=\id_{a_U}$. The functor $\eta$ sends the collection $(a_{U_+},
\alpha_{U_+,V_+})$ to the collection
\[(\Phi_{U_+,U}(a_{U_+}),\Phi_{U_+,U}(\alpha_{U_+,V_+})).\]
Here, $U\subseteq V$ are proper subsets of ${\bf n-1}$. Thus, objects of $\cP$
are triples
\[(b_{\bf n-1},B,\zeta:\eta(B)\to\delta(a_{\bf n-1})).\]
Here, $B=(b_{U_+},\beta_{U_+,V_+})$ is an object of
$\tl\limits_{\fB'(\nn)}\Phi$, $b_{\bf n-1}$ is an object of $\Phi({\bf n-1})$
and $\zeta$ is a morphism of the category $\tl\limits_{\fB(\nn)}\Phi $.

In more details, the objects of $\cP$ are collections 
\[(b_{\bf n-1}\in Ob(\Phi({\bf n-1})),\quad b_{U_+}\in Ob(\Phi(U_+)),
\quad \beta_{U_+,V_+},\quad\zeta_{U}),\]
where $U\subseteq V$ is a proper subset of ${\bf n-1}$,
$\beta_{U_+,V_+}:\Phi_{U_+,V_+}(b_{V_+})\to b_{U_+}$ is a morphism of the
category $\Phi(U_+)$, while
\[\zeta_{U}:\Phi_{U,{\bf n-1}}(b_{\bf n-1})\to\Phi_{U,U_+}(b_{U_+})\]
is a morphism of the category $\Phi(U)$. All these must satisfy two
compatibility conditions: The first conditions says that for any proper subsets
$U\subseteq V\subseteq W$ of $\bf n-1$, one has
\[\beta_{U_+,W_+}=\beta_{U_+,V_+}\circ \Phi_{U_+,V_+}(\beta_{V_+,W_+}).\]
The second condition is that for any proper subsets $U\subseteq V$ of $\bf n-1$,
the diagram
\[\xymatrix{\Phi_{U,{\bf n-1}}(b_{\bf n-1})\ar[rr]^{\Phi_{U,V}(\zeta_V)}
        \ar[dr]_{ \zeta_{U}} & & \Phi_{U,V_+}(b_{V_+})\ar[dl]^{\hspace{1em}
        \Phi_{U,U_+}(\beta_{U_+,V_+})}\\  & \Phi_{U,U_+}(b_{U_+}) & }\]
commutes. 

We are now in a position to construct the functors
\[\Gamma:\tl_{\fB(\n)}\Phi\to\cP\quad{\rm and}\quad
\Delta:\cP\to\tl_{\fB(\n)}\Phi.\]
The functor $\Gamma$ simply forgets some data. More precisely, take an object
$(a_S,\alpha_{S,T})$ of the category $\tl\limits_{\fB(\n)}\Phi$, where
$S\subseteq T$ are proper subsets of $\n$. Then, $\Gamma$ sends this object to
the collection
\[(b_{\bf n-1}\in Ob(\Phi({\bf n-1})),\quad b_{U_+}\in Ob(\Phi(U_+)),\quad
\beta_{U_+,V_+},\quad\zeta_{U}),\]
where $b_{\bf n-1}=a_{\bf n-1}$, $b_{U_+}=a_{U+}$, $\beta_{U_+,V_+}=
\alpha_{U_+,V_+}$ and $\zeta_U=\alpha_{U,U_+}^{-1}\circ\alpha_{U,{\bf n-1}}$.
Here, $U\subseteq V$ are proper subsets of ${\bf n-1}$.

The functor $\Delta$ is defined by
\[\Delta(b_{\bf n-1}, b_{U_+}, \beta_{U_+,V_+},\zeta_{U})=(a_S,\alpha_{S,T}),\]
where
\begin{equation*}a_S=\begin{cases} b_{\bf n-1},\quad{\rm if}\quad S={\bf n-1},
\\  b_{S},\quad {\rm if}\quad S\in\fB'(\nn),
\\  \Phi_{S,{\bf n-1}}(b_{\bf n-1}),\quad {\rm if}\quad S\in \fB(\nn)
\end{cases}
\end{equation*}
and
\begin{equation*}
\hspace{7.5em}\alpha_{S,T}=\begin{cases}\beta_{S,T}, \quad {\rm if} \quad
S\in \fB'(\nn)\\ \id, \quad {\rm if} \quad T\in \fB(\nn)\\ \id,\quad{\rm if}
\quad T={\bf n-1}\\ \zeta_{S}^{-1}\circ \Phi_{S_+,S}(\beta_{S^+,T}), \quad
{\rm if} \quad T\in \fB'(\nn) \ \ \& \ \ S\in \fB(\nn).
\end{cases}
\end{equation*}
We have used the fact that there are exactly four different cases if $U\subseteq
T$ are proper subsets of $\bf n$.

The composite $\Gamma\circ\Delta$ is the identity functor. On the other hand,
the composite $\Delta\circ\Gamma$ sends the collection $(a_S,\alpha_{S,T})$ to
the collection $(\tilde{a}_S,\tilde{\alpha}_{S,T})$. Here,
\begin{equation*}
\tilde{a}_{S}= \begin{cases} a_S, \quad {\rm if} \quad S={\bf n-1},\\
\id, \quad {\rm if} \quad  T={\bf n-1},\\
\id \quad {\rm if} \quad  T\in \fB(\nn)\\
\end{cases} 
\end{equation*}
and
\begin{equation*}
\hspace{8.5em}\tilde{\alpha}_{S,T}= \begin{cases} \alpha_{S,T}, \quad
{\rm if} \quad S\in \fB'(\nn),\\
\id, \quad {\rm if} \quad S\in \fB'(\nn),\\
\id, \quad {\rm if} \quad  T={\bf n-1}, \\
\id, \quad {\rm if} \quad T\in \fB'(\nn) \ \ \& \ \ S\in \fB(\nn).
\end{cases} 
\end{equation*}
Observe that
$$\xi_S=\begin{cases} \id, \quad {\rm if} \quad S={\bf n-1}\\
\id, \quad {\rm if} \quad S\in\fB'(\nn)\\
\alpha_{S,{\bf n-1}}, \quad {\rm if} \quad S\in\fB(\nn)
\end{cases}
$$
defines a natural transformation $\Delta\circ\Gamma\to\id$. This implies the
following proposition:

\begin{Pro}
	The following is a 2-pullback diagram in the 2-category of groupoids:
	\[\xymatrix{ \tl\limits_{\fB(\n)}\Phi \ar[r]\ar[d] & \tl\limits_{\fB'(\nn)}
		\Phi\ar[d]\\ \Phi({\bf n-1})\ar[r] & \tl\limits_{\fB(\nn)}\Phi.}\]
\end{Pro}

Yoneda's 2-lemma yields the following statement for 2-colimits:

\begin{Co}\label{2nab}
	The following is a 2-pushout diagram of groupoids
	\[\xymatrix{ \tc_{\fB(\nn)}\Psi\ar[r]\ar[d] & \tc_{\fB'(\nn)}\Psi\ar[d]
		\\  \Psi({\bf n-1})\ar[r] & \tc_{\fB(\n)}\Psi }\]
\end{Co}

\section{2-limits on subsets with codimension at most 3}\label{cod3}

Note that the condition of Theorem \ref{colim=2-colim} actually consists of
$2^{n-1}$ conditions. Each of these conditions itself has powersets in the
indecies. It is well-known that the latter can be reduced to $n^2$ many, since
the colimit over a poset can be reduced to the colimit over its subposet
consisting of elements with codimension at most 2.

In this section, we wish to simplify the condition of our theorem from something
exponential to something of magnitude $n^3$.
\newline

Let $0\leq k<n$ be a natural number. Denote by $\fB({\bf n,k})$ the elements of
$\fB(\n)$, which are subsets $S$ of ${\bf n}$, with $|S|\geq k$.

We can consider the restriction of any functor $\Phi:\fB^{op}\to\Grp$ to
$\fB({\bf n,k})$, which will again be denoted by $\Phi$.

\begin{Pro}\label{3212.13.04}
	Consider the obvious forgetful functor
	\begin{equation}\label{shez}
	\gamma_{k}:\tl_{\fB(\n)}\Phi\to\tl_{\fB({\bf n,k})}\Phi.
	\end{equation}
	\begin{itemize}
		\item[i)] The functor $\gamma_k$ is faithful.
		\item[ii)] The functor $\gamma_k$ is full and faithful if $k\leq n-2$.
        \item[iii)] The functor $\gamma_k$ is an equivalence of categories if
        $k\leq n-3$.
	\end{itemize}
\end{Pro}

\begin{proof}
    i) Take two objects $A=(a_S,\alpha_{S,T})$ and $B=(b_S,\beta_{S,T})$ of
    $\tl\limits_{\fB(\n)}\Phi$, where $S\subseteq T$ are proper subsets of $\n$.
    Assume $(f_U)$ and $(f'_U)$ are morphisms $A\to B$, such that $f=f'$ in
    $\tl_{\fB({\bf n,k})}\Phi$. In partucular, $\gamma_k(f_S)=\gamma_k(f'_S)$ if
    $|S|=n-1$. We have to show that $f_U=f'_U$ for all proper subsets
    $U\subsetneq\n$. For a given $U$, choose an $S$, such that $U\subseteq S$
    and $|S|=n-1$. We have a commutative diagram
	\[\xymatrix{ \phi_{U,S}(a_S)\ar[rr]^{\alpha_{U,S}}\ar[d]_{\phi_{U,S}(f_S)}
	& & a_U\ar[d]^{f_U}\\   \phi_{U,S}(b_S)\ar[rr]_{\beta_{U,S}} & & b_U.}\]
    This implies that
    \[f_U=\beta_{U,S}\circ\phi_{U,S}(f_S)\alpha_{U,S}^{-1}.\]
    Likewise, we have
    \[f'_U=\beta_{U,S}\circ\phi_{U,S}(f'_S)\alpha_{U,S}^{-1}.\]
    Since $f_S=f'_S$, we obtain $f_U=f'_U.$
	
    ii) By i), we need to show that $\gamma_k$ is full. We keep the notation as
    in i). Take a morphism $\gamma_k(A)\to\gamma_k(B)$. It is given by a
	collection of morphisms $f_S:a_S\to b_S$, where $|S|\geq k$. Moreover, if
	$S\subseteq T$ are two such subsets, we have a commutative diagram
	\[\xymatrix{ \phi_{S,T}(a_S)\ar[rr]^{\alpha_{S,T}}\ar[d]_{\phi_{S,T}(f_T)}
	& & a_S\ar[d]^{f_S} \\  \phi_{S,T}(b_T)\ar[rr]_{\beta_{S,T}} & & b_S.}\]
	We have to extend the definition of $f_S$ to include subsets $S\subseteq
    {\bf n}$ with $|S|<k$, in such a way, that the corresponding diagrams still
    commute for all $S\subseteq T$. This is done by induction on $i$, where
    $i=k-|S|$. If $i=0$, we already have such maps.
	
	Let $i>0$ and $U$ be a subset with $|U|=|S|+1$. Then $f_U$ is already
	defined by the induction assumption. We claim that the map
	\[\beta_{S,U}\circ\phi_{S,U}(f_U)\circ \alpha_{S,U}^{-1}:a_S\to b_S\]
	is independent of the choice of $U$. Assume $V$ is another subset of {\bf
		n}, such that $|V|=|S|+1$, $S\subseteq V$ and $V\not =U$. Take $W=V\cup U$.
	We have $|W|=|S|+2\leq n-1$. By the induction assumption, we have a
	commutative diagram
	\[\xymatrix{ \phi_{U,W}(a_W)\ar[rr]^{\alpha_{U,W}}\ar[d]_{\phi_{U,W}(f_W)}
	& & a_U\ar[d]^{f_U}\\ \phi_{U,W}(b_W)\ar[rr]_{\beta_{U,W}} & & b_U.}\]
	It follows that
	\begin{align*}\beta_{S,U}\circ\phi_{S,U}(f_U)\circ \alpha_{S,U}^{-1} & =
	\beta_{S,U} \phi_{S,U}(\beta_{U,W}\phi_{U,W}(f_W)\alpha_{U,W}^{-1})
	\alpha_{S,U}^{-1}
	\\  & = \beta_{S,U} \phi_{S,U}(\beta_{U,W})\phi_{S,U}( \phi_{U,W}(f_W))
	\phi_{U.W}(\alpha_{U,W}^{-1}) \alpha_{S,U}^{-1}\\
	& = \beta_{SW}\phi_{S,W}(f_W)\alpha_{S,W}^{-1}.
	\end{align*}
	Quite similarly, we obtain
	\[\beta_{S,V}\circ\phi_{S,V}(f_V)\circ \alpha_{S,V}^{-1}=\beta_{SW}
	\phi_{S,W}(f_W)\alpha_{S,W}^{-1}\]
	and the claim follows.
	
	It remains to show that the appropriate diagram commutes for $S\subseteq T$.
	Choose $U$ such that $S\subseteq U\subseteq T$ and $|U|=|S|+1$. The
	corresponding diagram commutes for $U\subseteq T$ by the induction
	assumption. On the other hand, we also have a commutative diagram for the
	pair $S\subseteq U$. This is due to the definition of $f_S$. We obtain a
    commutative diagram for the pair $S\subseteq T$, by gluing the two diagrams
    we just discussed.
	
    iii) By ii), we only need to show that $\gamma_k$ is essentially surjective.
    Take an object $B$ of $\tl\limits_{\fB(\bf n,k)}\Phi$. By definition, $B$ is
    a collection $(b_S,\beta_{S,T})$, where $|S|\geq k$ and $S\subseteq T$.
    These data satisfy the compatibility condition, which means that an
    appropriate diagram commutes, see the diagram in i). The task is to define
	an object $A=(a_S,\alpha_{S,T})$ of $\tl\limits_{\fB(\bf n)}\Phi$, such that
	$\gamma_k(A)\simeq B$. In fact, we will show $\gamma_k(A)=B$.
	
	If $|S|\geq k$, we set $a_S=b_S$ and $\alpha_{S,T}=\beta_{S,T}$. Let
	$|S|<k$. We may assume that the object $A_U$ is already defined for all $U$
	with $|U|>|S|$. Choose $i\in \n$ such that $i\not \in S$. We write $S\cup i$
	instead of $S\cup \{i\}$ and set
	\[a_S:=\phi_{S,S\cup i}(a_{S\cup i}).\]
	We wish to define the morphisms $\alpha_{S,T}:\phi_{S,T}(a_T)\to a_S$. There
	are several cases to consider:
	\item[$\bullet$] Let $T=S$. We put $\alpha_{S,S}:=\id$. We are left with
	$|T|>|S|$.
	\item[$\bullet$] Let $i\in T$. We define $\alpha_{S,T}:=\phi_{S,S\cup
		i}(\alpha_{S\cup i,T})$. 
	\item[$\bullet$] Assume $T\not =\n\setminus \{i\}$. Then $T\cup i$ is also a
	proper subset and we can put
	\[\alpha_{S,T}=\phi_{S,S\cup i}(\alpha_{S\cup i,T\cup i})\circ
	\phi_{S,T}(\alpha_{T,T\cup i})^{-1}.\]
	This is equivalent to saying that the following diagram commutes:
    \[\xymatrix{ \phi_{S,T\cup i}(a_{T\cup i})\ar[d]_{\phi_{S,S\cup i}
    (\alpha_{S\cup i,T\cup i})}\ar[rr]^{\phi_{S,T}(\alpha_{T,T\cup i})}
    \ & & \phi_{S,T}(a_T)\ar[d]^{\alpha_{S,T}}\\
	\phi_{S,S\cup i}(a_{S\cup i})  \ar[rr]_{\id} & & a_S.}\]
	\item[$\bullet$] It remains to consider the case $T=\n\setminus \{i\}.$
	Choose $j\in T$ such that $j\not \in S$. We set
	\[\alpha_{S,T}=\alpha_{S,S\cup j}\circ \phi_{S,S\cup j}(\alpha_{S\cup
		j,T}).\]
	One gets the following commutative diagram
	\[\xymatrix{ & & & \phi_{S,T}(a_T)\ar[d]^{\phi_{S,S\cup j}
	(\alpha_{S\cup j,T})}\\ \phi_{S,S\cup i\cup j}(a_{S\cup i\cup j})
	\ar[rrr]^{\phi_{S,S\cup j}(\alpha_{S\cup j, S\cup i\cup j})}
	\ar[d]_{\phi_{S,S\cup i}(\alpha_{S\cup i,S\cup i\cup j})} & & &
	\phi_{S,S\cup j}(a_{S\cup j})\ar[d]^{\phi_{S,S\cup j}(\alpha_{S,S\cup j})}
	\\  \Phi_{S,S\cup i}(a_{S\cup i})\ar[rrr]_{\id} & & & a_{S}}\]
	from the definition of $\alpha_{S,S\cup j}$. This definition is independent
	on the choice of $j$: In fact, chose $k$ instead of $j$. The set $S\cup
	i\cup j\cup k$ is still a proper subset of $\n$ (since $|S|<k\leq n-3$). One
	easily sees that both definitions of $\alpha_{S,T}$ equal the composite
	\[\phi_{S,S\cup i\cup j}(\phi (a_T))\xto{\phi_{S,S\cup i\cup
			j}(\alpha_{S\cup i \cup j})} a_{S\cup i \cup j}.\]
	This finishes the definition of $A$. Obviously $\gamma_k(A)=B$ and the
	result follows.
\end{proof}

 \begin{Co}\label{32131404}
     Let $n\geq 3$ be a natural number and $\Phi:\fB(n)\to\Grp$ a covariant
     strict $2$-functor. Recall that $\fB(n,k)$ is the subposet of $\fB(\n)$
     formed by subsets of cardinality $\geq k$. The canonical functor
     \[\tc_{\fB({\bf n,n-3})}\Phi\to\tc_{\fB(\n)}\Phi\]
     is an equivalence categories.
 \end{Co}

 \begin{proof}
     For groupoid $\cG$, we have
     \[\Hom_\Grp( \tc_{\fB(\n)}\Phi,\cG)\cong\tl_{\fB(\n)}\Hom_\Grp(\Phi,\cG).\]
	Thanks to Proposition \ref{3212.13.04}, we can replace the RHS as follows:
    \[\cong \tl_{\fB(\n, \bf n-3)}\Hom_\Grp(\Phi,\cG)\cong [\Hom_\Grp(
    \tc_{\fB(\n, {\bf n-3})}\Phi,\cG)\]
    The result follows.
\end{proof}
	
Recall that for subsets $V\subseteq{\bf n}$ and $U\subsetneq V$ condition
$A^V_U$ says that the canonical functor
\[{\sf colim}_{\fB(V:U)}\Phi\to\Phi(V)\]
is injective on objects.
Define
\[V_i={\bf n}\setminus\{i\},\quad U_i=\{i+1,\cdots,n\},\quad 2\leq i\leq n\]
and
\[V_{i,j}={\bf n}\setminus\{i,j\},\quad U_{i,j}=\{i+1,\cdots,n\}\setminus\{j\},
\quad 2\leq i <j\leq n.\]
We are in a position to state the following result:

\begin{Pro}\label{62,25.04}
    Let $n\geq3$ and $\Phi:\fB({\bf n})\to\Grp$ be a covariant strict
    $2$-functor, for which the conditions $A^{V_i}_{U_i}$ and
    $A^{V_{ij}}_{U_{ij}}$ hold for all  $2\leq i\leq n$ and $2\leq i<j\leq n$.
    Then
    \[\tc_{\fB(\n)}\Phi=\colim_{\fB(\bf n,n-3)}\Phi\]
\end{Pro}

\begin{proof}
    Denote by $\overline{ \Phi}$ the strict covariant $2$-functor
    $\fB(n)\to\Grp$, defined by
    \[{\overline{ \Phi}}(V)=\begin{cases} \Phi(V), \quad {\rm if} \quad
    Card(V)\geq n-3\\ \emptyset \quad {\rm if} \quad Card(V)< n-3.\end{cases}\]
    We have
    \[\tc_{\fB(\n)}\Phi= \tc_{\fB({\bf n, n-3})}\Phi\]
    by Corollary \ref{32131404}. Since both $2$-functors restricted on
    $\fB({\bf n, n-3)}$ are equal, it follows that
    \[\tc_{\fB({\bf n, n-3})}\Phi= \tc_{\fB({\bf n, n-3})}\overline{\Phi}\]
    Using Corollary \ref{32131404} again, we obtain
    \[\tc_{\fB({\bf n})}\Phi= \tc_{\fB({\bf n})}\overline{\Phi}.\]
    We claim that $\overline{\Phi}$ satisfies condition $A(n)$ from Theorem
    \ref{colim=2-colim}. Assuming this, we obtain
    \[\tc_{\fB({\bf n})}\Phi=\colim  _{\fB({\bf n})}\overline{\Phi} =
    \colim_{\fB(\bf n,n-3)}\overline{\Phi}=\colim_{\fB(\bf n,n-3)}{\Phi}.\]
    To prove the claim, take $V$, $k$, $U$ as in condition $A(n)$. We observer
    that condition $A(n)$ holds trivially for $\overline{\Phi}$ if $Card(V)\leq
    n-3$, because all Groupoids in the diagrams are empty. Assume $Card(V)=
    n-2$. Since $1\in V$, it follows that
    \[V=V_{i,j},\]
    where $2\leq i <j\leq n$. In the notation of Theorem \ref{colim=2-colim},
    $k=i-1$. Hence $U=U_{i,j}$. Quite similarly, if $Card(V)= n-1$ and $1\in V$,
    we see that $V=V_{i}$, where $2\leq i \leq n$. Thus, $k=i-1$ and $U=U_{i}.$
    By our assumptions, condition $A(n)$ holds for $\overline{\Phi}$.
\end{proof}

\end{document}